\documentclass[12pt,oneside,reqno]{amsart}
\usepackage{xcolor}
\usepackage{hyperref}

\usepackage{geometry}
\usepackage{amssymb,amsmath,amsthm}
\numberwithin{equation}{section}
\usepackage[english]{babel}
\usepackage{textcomp}
\usepackage{epstopdf}
\usepackage{inputenc}
\usepackage{graphicx}
\newcommand{\Aut}{\operatorname{Aut}}
\newcommand\numberthis{\addtocounter{equation}{1}\tag{\theequation}}

\definecolor{armygreen}{rgb}{0.29, 0.33, 0.13}
\definecolor{darkgreen}{rgb}{0.0, 0.2, 0.13}

\DeclareMathOperator{\Ran}{Ran}
\DeclareMathOperator{\Ker}{Ker}

\newtheorem{thm}{Theorem}[section]
\newtheorem{lem}[thm]{Lemma}

\newtheorem{Def}[thm]{Definition}
\newtheorem{prop}[thm]{Proposition}
\newtheorem{rem}[thm]{Remark}

\newcommand{\Rem}{\begin{rem} \rm}
	\newcommand{\bdfn}{\begin{Def} \rm}
		\newcommand{\edfn}{\end{Def}}

	\newcommand{\ba}{\begin{array}}
		\newcommand{\ea}{\end{array}}

\begin{document}
		\title[Generalized tri-circular projections on some vector-valued spaces]{Generalized tri-circular projections on some vector-valued spaces of analytic functions}
		
\author{Hemant Kumar}
\address{Hemant Kumar, Department of Applied Sciences, Indian Institute of Information Technology Allahabad, Prayagraj-211015, U.P., India.}
\email{hemantkr.math@gmail.com}

\author{Himanshu Kumar}
\address{Himanshu Kumar, Department of Applied Sciences, Indian Institute of Information Technology Allahabad, Prayagraj-211015, U.P., India.}
\email{himanshumath1507@gmail.com}

\author{Abdullah Bin Abu Baker}
\address{Abdullah Bin Abu Baker, Department of Applied Sciences, Indian Institute of Information Technology Allahabad, Prayagraj-211015, U.P., India.}
\email{abdullahmath@gmail.com}

\thanks{The first- and second-named authors gratefully acknowledge the Ministry of Education, New Delhi (India), for financial support, and IIIT Allahabad, India, for providing the resources and infrastructure to carry out this research.}

		
		
		\date{\today}
		
\begin{abstract}
In this paper, we characterize generalized tri-circular projections on the following three vector-valued spaces of analytic functions on the open unit disk $\mathbb{D}$: the Hardy space $H^p(\mathcal{K})$ for a complex separable Hilbert space $\mathcal{K}$ and $1 \le p < \infty,p \neq 2$; the space $H^\infty(E)$ of $E$-valued bounded analytic functions for a Banach space $E$ with trivial multiplier algebra; and the little Bloch spaces $\mathcal{B}_0(\mathbb{D}, E)$ and $\mathcal{B}_{\ast}(\mathbb{D}, E)$, for $E$ smooth, strictly convex, and reflexive.			
\end{abstract}
		
\maketitle
		
\thispagestyle{empty}
\section{Introduction}
A projection on a Banach space $X$ is a bounded linear operator $P$ such that $P^2 = P$. A projection $P$ is called \emph{contractive} (respectively, \emph{bi-contractive}) if $\|P\|= 1$ (respectively, $\|P\|= 1, \|I-P\|= 1$). Contractive projections are interesting in their own right, well beyond the obvious fact that they decompose a space into complemented subspaces. On spaces of analytic functions, the range of a contractive projection is often required to be a very rigid, geometrically distinguished subspace, and characterization of the subspaces arising in this way has been a recurring theme since the classical work on $L^p$ and $H^p$ spaces. A convenient way to obtain such projections is to average a surjective isometry (or a finite family of them) over the roots of unity or other unimodular weights, and this is what we do in this paper.

The simplest instance of this idea is a \emph{circular projection}: given a surjective linear isometry $T$ of order $n$ (i.e., $T^n=I$) on $X$, the operator
\[
P = \frac{1}{n}\sum_{k=0}^{n-1} T^k
\]
is a contractive projection onto the fixed-point subspace of $T$. Such circular projections and the associated ``$n$-circular'' decompositions of $X$ into eigenspaces of $T$ go back to classical results on isometry groups of function spaces.

Let $\mathbb{D}$ be the open unit disk and let $\mathbb{T}$ be the boundary of this in the complex plane $\mathbb{C}$. The variant we consider here is a more flexible one, in which the single averaging operator is replaced by a convex combination with genuinely different unimodular coefficients. Given a surjective isometry $T$ with $T^n = I$ and \emph{distinct} scalars $\lambda_1,\dots,\lambda_{n-1}\in \mathbb{T} \setminus\{1\}$, one can still recover the spectral projections of $T$ corresponding to the eigenvalues $1,\lambda_1,\dots,\lambda_{n-1}$ from the polynomial identity $(T-I)(T-\lambda_1 I)\cdots(T-\lambda_{n-1}I)=0$. This is the mechanism behind \emph{generalized $n$-circular projections}: a family $\{P_0, P_1, \dots, P_{n-1}\}$ of nonzero, pairwise distinct contractive projections summing to the identity, arising as the spectral projections of some surjective isometry $T = P_0+\lambda_1 P_1+\cdots+\lambda_{n-1}P_{n-1}$ for some choices of distinct $\lambda_1, \dots, \lambda_{n-1 }\in \mathbb{T} \setminus \{1\}$. For $n=2$, these are the well-studied bi-contractive class of \emph{generalized bi-circular projections}, introduced by Fo{\v{s}}ner, Ili{\v{s}}evi{\'c}, and Li \cite{fovsner2007g}, and since characterized on a wide range of function spaces, see for example,  \cite{botelho2009generalized, botelho2014isometries}. The current work concerns the next case, $n=3$, for some vector-valued spaces of analytic functions on the open unit disk, usually called \emph{generalized tri-circular projections}. For more related results, we refer the reader to \cite{abubaker2016structures, hosseini2022note, hosseini2021projections, ilivsevic2022generalized, lin2008generalized, maurya2024automorphisms}, and the references therein. 

In this article we consider three vector-valued spaces of analytic functions on the open unit disk $\mathbb{D}$: the vector-valued Hardy space $H^p(\mathcal{K})$ for a complex separable Hilbert space $\mathcal{K}$ and $1\le p <\infty$, $p\ne 2$; the vector-valued space $H^\infty(E)$ of bounded analytic functions with values in a Banach space $E$ with trivial multiplier algebra; and the vector-valued little Bloch space $B_0(\mathbb{D}, E)$ (together with its unitization $B_*(\mathbb{D}, E)$), for $E$ smooth, strictly convex, and reflexive. In each case, a classification of the surjective linear isometries is already available in the literature.

The paper is organized as follows. Section \ref{sec:prelim} fixes notation, recalls the definition of a generalized $n$-circular projection and the accompanying factorization Lemma, and states the three isometry-classification Theorems we will need. Section \ref{sec:hardy} treats $H^p(\mathcal{K})$; Section \ref{sec:boundedhardy} treats $H^\infty(E)$; Section \ref{sec:bloch} treats $B_0(\mathbb{D},E)$ and $B_*(\mathbb{D},E)$. We have in all sections the same 3-case splits: $\varphi$ the identity, $\varphi$ of order $2$, and $\varphi$ of order exactly $3$. The case of order $\varphi \geq 4$ is not possible, as it leads to a contradiction. 

\section{Preliminaries}\label{sec:prelim}

Throughout, $\Aut(\mathbb{D})$ is the group of disk automorphisms; every $\varphi \in \Aut(\mathbb{D})$ has order $1$, $2$, $3$, or is of infinite order as a self-map of $\mathbb{D}$ under composition. 

\subsection{Generalized $n$-circular projections}
\begin{Def} \label{def:gncp}
Let $X$ be a Banach space and let $n\ge 2$. A family $\mathcal{C}=\{P_0,P_1,\dots,P_{n-1}\}$ of nonzero, pairwise distinct nonzero projections on $X$ is called a \emph{generalized $n$-circular projection} with respect to a choice of distinct scalars $\lambda_1,\dots,\lambda_{n-1}\in\mathbb{T}\setminus\{1\}$ if
\[
P_0+P_1+\cdots+P_{n-1} = I \qquad\text{and}\qquad T:=P_0+\lambda_1 P_1+\cdots+\lambda_{n-1}P_{n-1}
\]
is a surjective linear isometry of $X$. If $n=2$, we say that $\mathcal C$ is a family of  \emph{generalized bi-circular projections} (in short, $GBPs$); and if $n=3$, a family of \emph{generalized tri-circular projections} (in short, $GTPs$).
\end{Def}

In \cite{vcuka2016generalized}, the authors have shown that the surjectivity of the isometry $T$ is not necessary. The point of Definition \ref{def:gncp} is that it converts a statement about a family of projections into a statement about a single isometry $T$ satisfying a polynomial identity, which is far more tractable given an explicit classification of the isometries of $X$. This conversion, and its converse, is made precise by the following elementary but essential Lemma, which we recall from \cite{abu2016generalized} and use repeatedly in this paper.

\begin{lem}\label{lem:factorization}
Let $X$ be a Banach space, and let $\mathcal{C} = \{P_0,P_1, P_2\}$ be a collection of nonzero distinct projections. Let $\lambda_1, \lambda_2$ be distinct complex numbers not equal to $1$. Then the following conditions are equivalent.
\begin{enumerate}
\item $\mathcal{C}$ is the family of generalized tri-circular projections and $T = P_0 + \lambda_1 P_1 + \lambda_2 P_2$.
\item The following holds: 
$ (T- I)(T -\lambda_1 I)(T - \lambda_2 I) =0 $ and
$$ P_0 = \frac{(T - \lambda_1 I)(T- \lambda_2 I)}{(1 - \lambda_1)(1 - \lambda_2)}, P_1 = \frac{(T - I)(T- \lambda_2 I)}{(\lambda_1 - 1)(\lambda_1 - \lambda_2)}, P_2 = \frac{(T - I)(T- \lambda_1 I)}{(\lambda_2 - 1)(\lambda_2 - \lambda_1)}. $$
\end{enumerate}
\end{lem}
An operator $P$ on a Hilbert space $\mathcal{H}$ is called orthogonal if $\Ran(P) \perp \Ker (P)$, where $\Ran(P)$ and $\Ker (P)$ denote the range space and kernel space of $P$, respectively. We now prove that $GTP$ on a Hilbert space is nothing but the orthogonal projection. To prove this, we need the following lemma.  
\begin{lem}\label{lem:kernel}
Let $\mathcal H$ be a Hilbert space and let $P_0,P_1,P_2$ be projections on $\mathcal H$ satisfying $P_0+P_1+P_2=I$ and $P_iP_j=0$ for $i\neq j$. Then for each $i\in\{0,1,2\}$,
\[
\Ker(P_i)=\bigoplus_{j\neq i}\Ran(P_j).
\]
\end{lem}
\begin{proof}
For $j\neq i$, $P_iP_j=0$ gives $\Ran(P_j)\subseteq\Ker(P_i)$. Conversely, if $v\in\Ker(P_i)$, then $v=P_0v+P_1v+P_2v$, and since $P_iv=0$ the $i$-th summand vanishes, so
\[
v=\sum_{j\neq i}P_jv\in\sum_{j\neq i}\Ran(P_j).
\]
The sum is direct because if we take $x \in \Ran(P_j) \cap \Ran (P_k)$ for $j \neq k$, then $x=P_j s = P_k t$ for some $s, t \in \mathcal{H}$. This implies $P_j^2s=P_j s=x = P_j P_k t=0$.
\end{proof}
\begin{prop}\label{prop:gtp-hermitian}
Let $\mathcal H$ be a Hilbert space and let $\mathcal C=\{P_0, P_1, P_2\}$ be a family of non zero distinct projections such that $P_0+ P_1+ P_2=I$ and $P_i P_j=0$ for $i \neq j$. Then $\mathcal{C}$ is a family of generalized tri-circular projections on $\mathcal H$ if and only if $P_i$ is orthogonal for $i=0,1,2$. 
\end{prop}
\begin{proof} ($\Longrightarrow$) Suppose $\mathcal C=\{P_0,P_1,P_2\}$ is a family of $GTPs$ on $\mathcal H$. Then $P_0+P_1+P_2=I$, $P_iP_j=0$ for $i\neq j$, and there exist distinct $\lambda_1,\lambda_2\in\mathbb{T} \setminus \{1\}$,
and a unitary operator $U$ on $\mathcal H$, such that $U=P_0+\lambda_1P_1+\lambda_2P_2$. 

Let $\lambda_0=1$. If $u\in\Ran(P_i)$ then $P_iu=u$ and, by Lemma~\ref{lem:kernel}, $P_ku=0$ for every $k\neq i$. Hence $Uu=\lambda_iu$ for all $u\in\Ran(P_i)$. We claim that $\Ran(P_i)\perp\Ran(P_j)$ for all $i\neq j$. To see this, fix $i\neq j$ and let $u\in\Ran(P_i)$, $v\in\Ran(P_j)$. Since $U$ is unitary, $\langle v,u\rangle=\langle Uv,Uu\rangle=\langle\lambda_jv,\lambda_iu\rangle = \lambda_j\overline{\lambda_i}\,\langle v,u\rangle$. Thus, $(\lambda_j\overline{\lambda_i} - 1)\langle v,u\rangle=0$. It follows that $\langle v,u\rangle=0$. Moreover, it follows from Lemma~\ref{lem:kernel} that $\Ran(P_i) \perp \Ker(P_i)$, $i=0,1,2$. Therefore, $P_i$ is an orthogonal projection. 

($\Longleftarrow$) Suppose $P_i$ is an orthogonal projection for each $i=0,1,2$. This implies that $P_i$ is Hermitian. Choose distinct $\lambda_1, \lambda_2 \in \mathbb{T} \setminus \{1\}$ and define $U = P_0 + \lambda_1 P_1 + \lambda_2 P_2$. We can easily verify that $U$ is unitary. Therefore, $\mathcal{C}$ is a family of generalized tri-circular projections. 
\end{proof}	
If $T$ is a surjective isometry of $X$ which satisfies the cubic identity $(T-I)(T-\lambda_1I)(T-\lambda_2I)=0$, we insert the explicit weighted-composition formula for $T$ (space by space below), and substitute into the cubic identity to obtain an operator Equation, which, evaluated on a small and well-chosen set of test functions, yields algebraic constraints on the disk-automorphism symbol $\varphi$ and the coefficient space operator. 

Let $\mathcal{C}=\{P_0,P_1,P_2\}$ be a family of $GTPs$ on any Banach space $X$ corresponding to a surjective isometry $T$. Then there exist distinct unimodular constants
$\lambda_1,\lambda_2\in\mathbb{T}\setminus\{1\}$ such that
$$
T=P_0+\lambda_1P_1+\lambda_2P_2.
$$
Consequently,
\begin{equation*}\label{eq1}
(T-I)(T-\lambda_1I)(T-\lambda_2I)=0,
\end{equation*}
which expands to
\begin{equation}\label{MainIdentity}
T^3-(1+l)T^2+(l+m)T-mI=0, \quad l=\lambda_1+\lambda_2,\; m=\lambda_1\lambda_2,
\end{equation}
which is the identity with which we work throughout Sections \ref{sec:hardy}--\ref{sec:bloch}.

Now, we recall the definitions of Hardy spaces $H^p(\mathcal{K})$, $H^\infty(E)$, and vector-valued little Bloch spaces $B_0(\mathbb{D},E)$ and $B_*(\mathbb{D},E)$. We also recall the form of surjective linear isometries on each of these spaces. 

\subsection{Vector-valued Hardy space $H^p(\mathcal{K})$}

Let $\mathcal{K}$ be a complex separable Hilbert space and $1\le p<\infty$. We denote by $H^p(\mathcal{K})$ the space of analytic $\mathcal{K}$-valued functions $f:\mathbb{D}\to \mathcal{K}$ with \begin{equation*} \|f\| = \sup_{0<r<1}\left(\frac{1}{2\pi}\int_0^{2\pi}\|f(re^{it})\|_{\mathcal{K}}^p \,dt\right)^{1/p}<\infty. \end{equation*}
For $p\ne 2$, every surjective linear isometry $T: H^p(\mathcal{K})\to H^p(\mathcal{K})$ has the form
\begin{equation}\label{Hp-isom}
(Tf)(z) = (\varphi'(z))^{\frac{1}{p}} U(f(\varphi(z))), \qquad f\in H^p(\mathcal{K}),\ z\in\mathbb{D},
\end{equation}
for some $\varphi\in\Aut(\mathbb{D})$ and some unitary operator $U$ on $\mathcal{K}$ \cite{lin1991isometries}.

\subsection{Vector-valued space $H^\infty(E)$}
Let $E$ be a complex Banach space whose multiplier algebra is trivial. We write $H^\infty(E)$ for the space of bounded analytic functions $f:\mathbb{D}\to E$ with 
$$
\|f\|_\infty = \sup_{z\in\mathbb{D}}\|f(z)\|_E.
$$   
\begin{Def}
    An operator $\mathrm{M}$ on a complex Banach space $X$ is called a \textit{multiplier} of $X$ if there is a function $a_\mathrm{M}: E_X \to \mathbb{C}$ such that $p \circ \mathrm{M}= a_{\mathrm{M}} (p) p$ for every $p \in E_X$, where $E_X$ denotes the set of extreme points of the unit ball of the dual space of $X$. The set of all multipliers of $X$ is denoted by $\text{Mult}(X)$. 
\end{Def} 

Examples of Banach spaces with trivial multipliers include Hilbert spaces and uniformly convex and uniformly smooth spaces. For more details, see the book by Behrends \cite{behrends2006m}.  

The isometry group of  $H^\infty(E)$ was first described by Lin \cite{lin1990isometries} under the hypothesis that $E$ is uniformly convex and uniformly smooth; Cambern and Jarosz \cite{cambern1990multipliers} subsequently established the same conclusion under the strictly weaker hypothesis that $E$ has trivial multiplier (i.e. $\text{Mult}(E)= \mathbb{C}$). Every surjective linear isometry $T:H^\infty(E)\to H^\infty(E)$ has the form
\begin{equation}\label{Hinf-isom}
(Tf)(z) = J\big[f(\varphi(z))\big], \qquad f\in H^\infty(E),\ z\in\mathbb{D},
\end{equation}
for some $\varphi\in\Aut(\mathbb{D})$ and some surjective linear isometry $J$ of $E$
\subsection{Vector-valued little Bloch space}

Let $E$ be a complex Banach space. The vector-valued little Bloch space $B_0(\mathbb{D},E)$ consists of analytic $f:\mathbb{D}\to E$ vanishing at $0$ for which
\[
\lim_{|z|\to 1^-}(1-|z|^2)\|f'(z)\|_E = 0,
\]
equipped with the norm 
$$\|f\|_{\mathcal B} = \sup_{z\in\mathbb{D}}(1-|z|^2)\|f'(z)\|_E.
$$ 

The unitized space $B_*(\mathbb{D},E)$ drops the vanishing-at-$0$ condition and carries the norm 
$$
\|f\|_{B_*} = \|f(0)\|_E + \sup_{z\in\mathbb{D}}(1-|z|^2)\|f'(z)\|_E.
$$ 

As Banach spaces, $B_*(\mathbb{D},E) \cong B_0(\mathbb{D},E)\oplus_1 E$ isometrically, via $f\mapsto (f-f(0),f(0))$.

For a smooth, strictly convex, and reflexive Banach space $E$, Botelho and Jamison \cite{botelho2014isometries} showed that $T: B_0 (\mathbb{D}, E)\to B_0(\mathbb{D}, E)$ is a surjective linear isometry if and only if there exist $\varphi\in\Aut(\mathbb{D})$ and a surjective linear isometry $V$ of $E$ with
\begin{equation}\label{B0-isom}
(Tf)(z) = V\big[f(\varphi(z)) - f(\varphi(0))\big], \qquad f\in B_0(\mathbb{D},E),\ z\in\mathbb{D}.
\end{equation}
Passing to $B_*(\mathbb{D},E)$ via the direct-sum decomposition above, the surjective isometries of $B_*(\mathbb{D},E)$ are exactly the operators
\begin{equation}\label{Bstar-isom}
(Tf)(z) = U(f(0)) + V\big[f(\varphi(z)) - f(\varphi(0))\big],
\end{equation}
for surjective linear isometries $U,V$ of $E$ and $\varphi\in\Aut(\mathbb{D})$.
%In all three settings above, the classification Theorem expresses a surjective isometry $T$ as a weighted composition operator built from a disk automorphism $\varphi$ and one (or two) isometries of the coefficient space. We insert the known form of the surjective isometries of each space into the cubic operator identity satisfied by $T$ and by evaluating the resulting identity on a small family of test functions, we isolate the constraints on $\varphi$ and on the coefficient-space isometry.
%%%%%%%%%%%%%%%%%%%%%%%%%%%%%%%%%%%%%%%%%%%%%%%%%%%%%%%%%%%%%%%%%%%%%%
\section{Generalized tri-circular projections on $H^{p}(\mathcal{K})$} \label{sec:hardy}
We first consider the $GTPs$ on $H^{p}(\mathcal{K})$ for $1 \leq p < \infty, p \neq 2$. We obtain the following characterization.
\begin{thm} \label{main1}
Let $\mathcal{C} = \{P_0, P_1, P_2\}$ be a family of generalized tri-circular projections on $H^{p}(\mathcal{K}), 1 \leq p < \infty, p \neq 2$. Then one of the assertions holds:
\begin{enumerate}
\item $\varphi (z) = z$. In this case, $P_if(z) = P_i^{\mathcal{K}}(f(z)), i=0,1,2$, where $P_i^{\mathcal{K}}$ is an orthogonal projection on the Hilbert space $\mathcal{K}$.

% $\mathcal{C}$ becomes a family of tri-circular projections, and the point spectrum of $U$ is $\sigma_p = \{1, \lambda_1, \lambda_2\}$.

\item $\varphi(z) \neq z, \varphi^2(z) =z$. In this case, one of the members from $\mathcal{C}$ becomes a generalized bi-circular projection.

\item $\varphi(z) \neq z, \varphi^2(z) \neq z, \varphi^3(z) =z$. Here, $\lambda_1$ and $\lambda_2$ are cube roots of unity, and $U^3= I_{\mathcal{K}}$, where $I_{\mathcal{K}}$ denotes the identity operator on $\mathcal{K}$.
\end{enumerate}
\end{thm}

\begin{proof}
Let $\mathcal{C} = \{P_0, P_1, P_2\}$ be a family of $GTPs$ on $H^{p}(\mathcal{K})$, $1 \leq p < \infty, p \neq 2$.
Using the form of isometry given in \eqref{Hp-isom}, we have
\begin{align*}
(Tf)(z) &= (\varphi'(z))^\frac{1}{p} U (f(\varphi(z))), 
 \\
(T^2 f)(z)&= (\varphi'(z) \varphi'( \varphi(z)))^\frac{1}{p} U^2(f(\varphi^2(z))),  \\
 (T^3 f)(z)) &= ( \varphi'(z) \varphi'(\varphi(z)) \varphi'(\varphi^2(z)))^\frac{1}{p} U^3 (f(\varphi^3(z))).  
\end{align*}
Now Equation \ref{MainIdentity} reduces to
\begin{align} \label{me}
((\varphi^3)'(z))^\frac{1}{p} U^3 (f(\varphi^3(z))) - (1+l)((\varphi^2)'(z))^\frac{1}{p} U^2 (f(\varphi^2(z)))   & \nonumber \\ 
+ (l+m) (\varphi'(z))^\frac{1}{p} U (f(\varphi(z))) - m f(z) = 0.  
\end{align}
Thus, we have only three possible cases:
\begin{enumerate}
\item  $\varphi(z) = z,$
\item $\varphi(z) \neq z,\varphi^2(z) = z,$ 
\item $\varphi(z) \neq z,\varphi^2(z) \neq z,\varphi^3(z) = z$.
\end{enumerate}

\subsection*{Case(I) $\varphi(z) = z$.}
Now, Equation \ref{me} takes the form
\begin{align*}
 U^3 f(z) - (1+l)U^2 f(z) + (l+m) U f(z) - m f(z) =0.     
\end{align*}
This implies 
\begin{equation*}
 (U-I_{\mathcal{K}}) (U- \lambda_1 I_{\mathcal{K}}) (U- \lambda_2 I_{\mathcal{K}})f(z) = 0.  
\end{equation*}

Then from Lemma \ref{lem:factorization}, there exists a family $\{ P_0^{\mathcal{K}}, P_1^{\mathcal{K}}, P_2^{\mathcal{K}} \}$ of $GTPs$ on $\mathcal{K}$ such that $P_0^{\mathcal{K}}+ \lambda_1 P_1^{\mathcal{K}} + \lambda_2 P_2^{\mathcal{K}} =U$. Using Proposition \ref{prop:gtp-hermitian}, each $P_i^{\mathcal{K}}$ is an orthogonal projection on $\mathcal{K}$. Thus, we have $P_if(z) = P_i^{\mathcal{K}}(f(z))$ for $i =0,1,2$.

% Therefore, the point spectrum of $U$ is $\sigma_p = \{1, \lambda_1, \lambda_2\}$. Moreover, $\mathcal{C}$ is a family of tri-circular projections.

\subsection*{Case(II) $\varphi(z) \neq z, \varphi^2(z) = z$.}
In this case, Equation \ref{me} converts to
\begin{align*} \label{mi2}
 U^3 [ ( \varphi'(z))^\frac{1}{p} f(\varphi(z))] - (1+l)U^2 f(z) + (l+m) U[ (\varphi'(z))^\frac{1}{p} f(\varphi(z)) ] - m f(z) =0. \numberthis    
\end{align*}
Again, for a fixed nonzero vector $v \in \mathcal{K}$, we choose the functions $f_1(z)= v$ and  $f_2(z) = z v$  for all $z \in \mathbb{D}$. Evaluating Equation \ref{mi2} for $f_1$ and $f_2$, we get the following Equations, respectively.
\begin{equation} \label{f4}
[(\varphi)'(z)]^{\frac{1}{p}} 	U^3v -(1+l) U^2 v + (l+m) [\varphi'(z)] ^{\frac{1}{p}} Uv -mv =0,
\end{equation}
\begin{equation} \label{f5}
[(\varphi)'(z)]^{\frac{1}{p}} \varphi(z) 	U^3v -(1+l)z U^2v + (l+m) [\varphi'(z)] ^{\frac{1}{p}} \varphi(z) Uv -m zv =0.
\end{equation}			
Multiplying Equation \eqref{f4} by $\varphi(z)$ and then subtracting it from \eqref{f5}, we arrive at
\begin{equation*} 
(z - \varphi(z)) [ -(1+l) U^2v -mv]=0.
\end{equation*}
Since $\varphi(z) \neq z$, we get  
\begin{equation} \label{f6}
mv = -(1+l) U^2v.
\end{equation}			
Again by Equation \eqref{f4}, we have 			
\begin{equation*} 
(1+l) U^2 v = [(\varphi)'(z)]^{\frac{1}{p}} 	U^3v  + (l+m) [\varphi'(z)] ^{\frac{1}{p}} Uv -mv.
\end{equation*}			
Putting this in Equation \eqref{f5}, we get
\begin{equation*} 
(\varphi(z)-z) [ (\varphi)'(z)]^{\frac{1}{p}}	[U^3v  + (l+m) Uv]  =0.
\end{equation*}			
Here $\varphi(z) \neq z$, thus			
\begin{equation} \label{f7}
U^3v  =- (l+m)Uv.
\end{equation}		
Simplifying Equations \ref{f6} and \ref{f7}, we obtain $(1+l)(l+m) -m =0$, i.e. $(\lambda_1+1)(\lambda_2+1)(\lambda_1+ \lambda_2) =0$. This implies that $\lambda_1= - \lambda_2$ or $\lambda_2=-1$ or $\lambda_1=-1$. Therefore, $P_0$ or $P_1$ or $P_2$, respectively, becomes a $GBP$.

\subsection*{Case (III)} $\varphi(z) \neq z,\varphi^2(z) \neq z, \varphi^3(z) = z$.

In this case, Equation \ref{me} converts to
\begin{align*} \label{eq3.12}
 U^3  f(z) - (1+l)U^2[ ((\varphi^2)'(z) )^\frac{1}{p} f(\varphi^2(z)) ] + (l+m) U[ (\varphi'(z))^\frac{1}{p} f(\varphi(z)) ] - m f(z) =0. \numberthis    
\end{align*}
For a fixed nonzero vector $v \in \mathcal{K}$, evaluating Equation \eqref{eq3.12} for the functions $f_1(z) = v, f_2(z)= zv $ and $f_3(z) = z^2v$ for all $z \in \mathbb{D}$, we obtain the following Equations, respectively.         
\begin{equation} \label{f1}
U^3v -(1+l) [(\varphi^2)'(z)]^{\frac{1}{p}} U^2v +(l+m)[\varphi'(z)] ^{\frac{1}{p}} Uv -mv =0.
\end{equation}	
\begin{equation}\label{f2}
z U^3v  -(1+l) [(\varphi^2)'(z)]^{\frac{1}{p}} \varphi^2(z)U^2v +(l+m) [\varphi'(z)] ^{\frac{1}{p}} \varphi(z)Uv -mzv =0.
\end{equation}			
\begin{equation} \label{f3}
z^2 U^3v  -(1+l) [(\varphi^2)'(z)]^{\frac{1}{p}} (\varphi^2(z))^2U^2v +(l+m) [\varphi'(z)] ^{\frac{1}{p}} (\varphi(z))^2Uv -mz^2v =0.
\end{equation}
Multiplying Equations \eqref{f1} and \eqref{f2} by $z$, then subtracting Equation \eqref{f3} from \eqref{f2} and Equation \eqref{f2} from \eqref{f1} we get 
\begin{equation*}\label{c11}
(1+l) [(\varphi^2)'(z)]^{\frac{1}{p}}[z - \varphi^2(z)] = (l+m)[\varphi'(z)] ^{\frac{1}{p}} [z- \varphi(z)], 
\end{equation*}
and
\begin{equation*} \label{c2}
(1+l) [(\varphi^2)'(z)]^{\frac{1}{p}} \varphi^2(z)[z - \varphi^2(z)] = (l+m)[\varphi'(z)] ^{\frac{1}{p}} \varphi(z)[z- \varphi(z)] .
\end{equation*}
It follows that
\begin{equation*}
(l+m)[\varphi'(z)] ^{\frac{1}{p}} (z-\varphi(z))(\varphi(z) - \varphi^2(z)) =0.
\end{equation*}
Since $(\varphi'(z))^{\frac{1}{p}} \neq 0$, it follows that			
\begin{equation*}
(l+m) (z-\varphi(z))(\varphi(z) - \varphi^2(z)) =0.
\end{equation*}
Since, $z \neq \varphi(z)$, we have $l+m =0$, i.e. $\lambda_1 + \lambda_2 + \lambda_1 \lambda_2 =0$. In similar ways $1+l =0$, i.e., $1 +	\lambda_1 + \lambda_2 =0$. This implies that $\lambda_1$ and $\lambda_2$ are cube roots of unity. Moreover, $U^3 = I_{\mathcal{K}}$. 

This completes the proof of the Theorem.
\end{proof}
%%%%%%%%%%%%%%%%%%%%%%%%%%%%%%%%%%%%%%%%%%%%%%%%%%%%%%%%%%%%%%%%%%%%%%
\section{Generalized tri-circular projections on $H^\infty_E$} \label{sec:boundedhardy}

We now turn our attention to the characterization of $GTPs$ on $H^\infty_E$. The following result gives the precise form of such projections on this space.

\begin{thm} \label{main2}
Let $\mathcal{C} = \{P_0, P_1, P_2\}$ be a family of generalized tri-circular projections on $H^\infty_E$. Then one of the following holds: 
\begin{enumerate}
\item $\varphi(z)=z$. In this case, $P_if(z) = P_i^{E}(f(z)), i=0,1,2$, and the collection $\mathcal{C}_E =\{ P_0^{E}, P_1^{E}, P_2^{E}\}$ is a family of $GTPs$ on $E$.
\item $\varphi(z)\neq z, \varphi^2(z)=z$.  In this case, $\mathcal{C}$ reduces to a family of generalized bi-circular projections.
\item $\varphi(z) \neq z, \varphi^2(z)\neq z$ and $\varphi^3(z) = z$. Then $\lambda_1$, $\lambda_2$ are cube roots of unity and $J^3=I_E$, where $I_E$ is the identity operator on $E$.
\end{enumerate}
\end{thm}
\begin{proof}
Let $\mathcal{C}=\{P_0,P_1,P_2\}$ be a family of $GTPs$
on $H^\infty_E$ corresponding to a surjective isometry $T$. Using \eqref{Hinf-isom}, a direct computation yields
$$
(T^k f)(z)=J^k\big(f(\varphi^k(z))\big), \qquad k=1,2,3.
$$
Substituting these expressions into \eqref{MainIdentity}, we obtain
\begin{equation}\label{eq3}
J^3 f(\varphi^3(z)) -(1+l)J^2 f(\varphi^2(z)) +(l+m)J f(\varphi(z)) -mf(z)=0,
\end{equation}
for all $f\in H^\infty_E$ and all $z\in\mathbb{D}$.

We have only three possible cases:
\begin{enumerate}
\item $\varphi(z)=z$,
\item $\varphi(z)\neq z, \varphi^2(z)=z$,
\item $\varphi(z) \neq z, \varphi^2(z)\neq z$ and $\varphi^3(z) = z$,
\end{enumerate}
\subsection*{Case (I) $\varphi(z)=z$.} We obtain, $(J-I_{E}) (J- \lambda_1 I_{E}) (J- \lambda_2 I_{E})f(z) = 0$. Using Lemma \ref{lem:factorization}, there exist $GTPs$, $P_0^{E}, P_1^{E}$ and $P_2^{E}$ on $E$ such that $P_0^{E} + \lambda_1 P_1^{E} + \lambda_2 P_2^{E} =J$. Consequently, $P_if(z) = P_i^{E}(f(z)), i =0,1,2$.

\subsection*{Case (II) $\varphi(z)\neq z, \varphi^2(z)=z$.}
In this case, Equation \ref{eq3} converts to
\begin{align*} \label{eq4.2}
J^3 f(\varphi(z)) -(1+l)J^2 f(z) +(l+m)J f(\varphi(z)) -mf(z)=0. \numberthis    
\end{align*}
Again, for a fixed nonzero vector $v \in E$, we choose the functions $f_1(z)= v$ and  $f_2(z) = z v$  for all $z \in \mathbb{D}$. Evaluating Equation \ref{eq4.2} for $f_1$ and $f_2$, we get the following Equations, respectively.
\begin{equation} \label{eqn4.3}
 	J^3v -(1+l) J^2 v + (l+m) J v - mv =0,
\end{equation}
\begin{equation} \label{eqn4.4}
 \varphi(z) 	J^3v -(1+l)z J^2 v + (l+m)  \varphi(z) Jv -m z v =0.
\end{equation}			
Multiplying Equation \eqref{eqn4.3} by $\varphi(z)$ and then subtracting it from \eqref{eqn4.4}, we get
\begin{equation*} 
(z - \varphi(z)) [ -(1+l) J^2v -mv]=0.
\end{equation*}
Since $\varphi(z) \neq z$, we get  
\begin{equation} \label{eqn4.5}
mv = -(1+l) J^2v.
\end{equation}			
Again by Equation \eqref{eqn4.3}, we have 			
\begin{equation} \label{eqn4.6} 
J^3v  = -(l+m) Jv. 
\end{equation}			
Simplifying Equations \ref{eqn4.5} and \ref{eqn4.6}, we obtain $((1+l)(l+m) -m)Jv =0 \implies (1+l)(l+m) -m =0$, i.e. $(\lambda_1+1)(\lambda_2+1)(\lambda_1+ \lambda_2) =0$. This implies that $\lambda_1= - \lambda_2$ or $\lambda_2=-1$ or $\lambda_1=-1$. Therefore, $P_0$ or $P_1$ or $P_2$, respectively, becomes a $GBP$.

\subsection*{Case (III) $\varphi(z) \neq z, \varphi^2(z)\neq z,\varphi^3(z) = z$.}
Evaluating \eqref{eq3}, and using constant vector-valued functions
$f(z)=v$ with $v\in E$, we obtain
\begin{equation} \label{eq2a}
    J^3-(1+l)J^2+(l+m)J-mI_E=0,
\end{equation}
where $I_E$ is the identity operator on $E$.

We can rewrite Equation \eqref{eq3} as following: 
\begin{equation*} \label{eq4}
(1+l) J^{2} [ f(\varphi^3(z)) -f(\varphi^{2}(z))] + (l+m) J [f(\varphi(z)) - f(\varphi^3(z))] + m [f(\varphi^3(z))-f(z)] =0. \numberthis
\end{equation*} 
Equation \eqref{eq4} evaluated for functions $f(z)= zv$ and $f(z)= z^2v$, where $v (\neq 0) \in E$  produces the following:
\begin{equation*}\label{Eq6}
(1+l) (z- \varphi^2(z)) J^{2} v = (l+m) (z-\varphi(z))Jv, \numberthis
\end{equation*}
\begin{equation*}\label{Eq7}
(1+l) (z^2- (\varphi^2(z))^2) J^{2} v = (l+m) (z^2-(\varphi(z))^2)Jv. \numberthis
\end{equation*}
Using \eqref{Eq6} in \eqref{Eq7} and eliminating $(1+l)$ we get
\begin{equation*}
(l+m) (z-\varphi^2(z)) (z + \varphi^2(z))Jv = (l+m) (z-\varphi(z)) (z+ \varphi(z))Jv. 
\end{equation*}
\begin{equation*}
\implies (l+m) (z-\varphi(z)) [ \varphi^2(z)-\varphi(z)]Jv = 0.
\end{equation*}
Here $z \neq \varphi(z) \neq \varphi^2(z)$, $v$ is an arbitrary element of $E$ and $J$ is an isometry on $E$, so $Jv \neq 0$. Therefore we have
\begin{equation*}
l+m=0 \text{ i.e. } \lambda_1 + \lambda_2 + \lambda_1 \lambda_2 =0.
%\label{Eq8}
\end{equation*}
Again, using \eqref{Eq6} in \eqref{Eq7} and eliminating $(l+m)$ we get
\begin{equation*}
(1+l) (z-\varphi^2(z)) (z + \varphi^2(z))J^2 v = (1+l) (z-\varphi^2(z)) (z+ \varphi(z))J^2v. 
\end{equation*}
\begin{equation*}
\implies (1+l) (z-\varphi^2(z)) [\varphi^2(z)- \varphi(z)]J^2 v = 0. 
\end{equation*}
Again, since $z \neq \varphi(z) \neq \varphi^2(z)$, $v$ is an arbitrary element of $E$ and $J^2$ is an isometry on $E$, so $J^2v \neq 0$. Therefore, we have
\begin{equation*}
1+l=0, \; \text{i.e.} \; 1+ \lambda_1 + \lambda_2 =0.
%\label{Eq9}
\end{equation*}
Combining this with the above, we get $\lambda_1, \lambda_2$ are cube roots of unity. Moreover, $J^3=I_E$, the identity operator on $E$.

This completes the proof of the Theorem.
\end{proof}
%%%%%%%%%%%%%%%%%%%%%%%%%%%%%%%%%%%%%%%%%%%%%%%%%%%%%%%%%%%%%%%%%%%%%%%%%%
\section{Generalized tri-circular projections on vector-valued little Bloch spaces} \label{sec:bloch}
In this section, we characterize the $GTPs$ on the space $\mathcal{B}_0(\mathbb{D},E)$. Using this characterization, we then derive the form of $GTPs$ on the space $\mathcal{B}_\ast(\mathbb{D},E)$.		
\begin{thm} \label{main3}
Let $\mathcal{C} = \{P_0, P_1, P_2\}$ be a family of generalized tri-circular projections on  $\mathcal{B}_0(\mathbb{D}, E)$. Then one of the following holds: 

\begin{enumerate}
\item $\varphi(z)=z$. In this case, $P_if(z) = P_i^{E}(f(z)), i=0,1,2$, and the collection $\mathcal{C}_E =\{ P_0^{E}, P_1^{E}, P_2^{E}\}$ is a family of $GTPs$ on $E$.
\item $\varphi(z)\neq z, \varphi^2(z)=z$.  In this case, $\mathcal{C}$ reduces to a family of generalized bi-circular projections.
\item $\varphi(z) \neq z, \varphi^2(z)\neq z$ and $\varphi^3(z) = z$. Then $\lambda_1$, $\lambda_2$ are cube roots of unity and $V^3=I_E$.
\end{enumerate}
\end{thm}
\begin{proof}
Let $\mathcal{C} = \{P_0, P_1, P_2\}$ be a family of $GTPs$ on $\mathcal{B}_0(\mathbb{D}, E)$ with respect to a surjective isometry $T$ and scalars $\lambda_1, \lambda_2 \in \mathbb{T} \setminus \{1\}$. From \ref{B0-isom} the isometry $T$ has the form
\begin{equation*}
    T f(z)= V [(f \circ \varphi) (z)- (f\circ \varphi)(0)],
\end{equation*}
for every $f \in \mathcal{B}_0(\mathbb{D}, E)$ and $z\in \mathbb{D}$, where $V$ is a surjective isometry on $E$ and $\varphi$ is the disk automorphism.  
			
A simple computation gives
$$ (T^2f)(z)= V^2 [ (f\circ \varphi^2)(z)- (f\circ \varphi^2)(0) ], $$
$$ (T^3f)(z)= V^3 [ (f\circ \varphi^3)(z)- (f\circ \varphi^3)(0) ].$$
Putting all these values in \eqref{MainIdentity} we get 
\begin{align*}
V^3 [ (f\circ \varphi^3)(z)- (f\circ \varphi^3)(0) ]- (1+l)V^2 [ (f\circ \varphi^2)(z)- (f\circ \varphi^2)(0) ]& \noindent \\ + (l+m)V [(f \circ \varphi) (z)- (f\circ \varphi)(0)]-m f(z)=0,
\end{align*}
for every $f \in \mathcal{B}_0(\mathbb{D}, E)$ and $z\in \mathbb{D}$. By differentiating this, we obtain the following
\begin{align*}
V^3 [ (f' \circ \varphi^3)(z)( \varphi^3)' (z) ]- (1+l)V^2 [ (f' \circ \varphi^2)(z) (\varphi^2)' (z)]& \noindent \\ + (l+m)V [(f' \circ \varphi) (z) \varphi ' (z)]-m f'(z)=0.
\end{align*}
\begin{align*} \label{Eq3.3}
\implies
 V^3 [ f'  (\varphi^3(z))( \varphi^3)' (z) ]- (1+l)V^2 [ f'( \varphi^2(z)) (\varphi^2)' (z)]& \noindent \\ + (l+m)V [(f'( \varphi (z)) \varphi ' (z)]-m f'(z)=0. \numberthis
\end{align*}
Thus, we have the following three cases: 
\begin{enumerate}
\item $\varphi(z)=z$,
\item $\varphi(z)\neq z, \varphi^2(z)=z$,
\item $\varphi(z) \neq z, \varphi^2(z)\neq z$ and $\varphi^3(z) = z$.
\end{enumerate}
\subsection*{Case (I) $\varphi(z)=z$.} 
From Equation \eqref{Eq3.3} we have the following. 
\begin{equation}
    (V^3 - (1+l)V^2  + (l+m)V-m I_E) f'(z)=0.
\end{equation} \label{eq4.20}
Choosing a function $f'(z)=v$, for a $v \in E$ and evaluating Equation \eqref{eq4.20} we get
\begin{equation*}
    (V^3 - (1+l)V^2  + (l+m)V-m I_E) v =0.
\end{equation*}
Since $v$ is arbitrary, we have the following. 
\begin{equation*}
V^3 -(1+l) V^2 + (l+m)V -m I_{E} = 0 \text{ or } (V - I_{E})(V - \lambda_1 I_{E})(V - \lambda_2 I_{E}) = 0.
\end{equation*}
Using Lemma \ref{lem:factorization}, there exist $GTPs$, $P_0^{E}, P_1^{E}$, and $P_2^{E}$ on $E$ such that $P_0^{E}+\lambda_1P_1^{E} +\lambda_2P_2^{E}=V$. Moreover, $P_if(z)=P_i^{E}(f(z))$, $i=0,1,2$.

\subsection*{Case (II) $\varphi(z)\neq z, \varphi^2(z)=z$.} 

In this case, Equation \eqref{Eq3.3} becomes
\begin{equation} \label{eq4.13}
(V^3+ (l+m)V)\varphi ' (z) f'( \varphi (z)) - ((1+l)V^2+m I_E) f'(z)=0.
\end{equation}
Let $v \in E$ be a non-zero element. Now we choose functions $f_1'(z)= v, f_2'(z)= z v$ and evaluating Equation \eqref{eq4.13} for a $z_0 \neq 0$, we get 
\begin{equation} \label{eq4.14}
 [(V^3+ (l+m)V)\varphi ' (z) - ((1+l)V^2+m I_E)] v=0,   
\end{equation}
\begin{equation} \label{eq4.15}
[(V^3+ (l+m)V)\varphi ' (z) \varphi(z) - ((1+l)V^2+m I_E) z] v=0.
\end{equation}
Multiplying by $\varphi(z)$ in Equation \eqref{eq4.14} and substracting from Equation \eqref{eq4.15} we get
\begin{equation*}
 - ((1+l)V^2+m I_E) (z-\varphi(z)) v=0.   
\end{equation*}
Since $\varphi(z) \neq z$, so $$(1+l)V^2(v)+mv=0. $$
Again multiplying by $z$ in Equation \eqref{eq4.14} and substracting from Equation \eqref{eq4.15} we get
\begin{equation*}
    (V^3+ (l+m)V)(\varphi (z)-z) \varphi'(z)v=0.
\end{equation*}
Since $\varphi(z) \neq z$, so $$V^3(v)+ (l+m)V(v)=0. $$
Combining these two Equations, we found that 
\begin{equation}
    (-(1+l)(l+m)+m) V(v)=0.
\end{equation}
This gives $(1+l)(l+m) -m =0$, i.e. $(\lambda_1+1)(\lambda_2+1)(\lambda_1+ \lambda_2) =0$, which implies that $\lambda_1= - \lambda_2$ or $\lambda_2=-1$ or $\lambda_1=-1$. Therefore, $P_0$ or $P_1$ or $P_2$, respectively, becomes a $GBP$.
\subsection*{Case (III) $\varphi(z) \neq z, \varphi^2(z)\neq z, \varphi^3(z) = z$.} 

Using this in Equation \eqref{Eq3.3}, we get
\begin{equation} \label{Eq3.4}
    [V^3-m I_E] f'  (z)- (1+l)V^2( f'( \varphi^2(z)) (\varphi^2)' (z)) \noindent  + (l+m)V (f'( \varphi (z)) \varphi'(z))=0.
\end{equation}
Let $ v \in E$ be a norm one element. Now we choose functions as followings:
\begin{enumerate}
    \item $f_1'(z)= v$
    \item $f_2'(z)= z v$
    \item $f_3'(z)= z^2v$, for every $z \in \mathbb{D}$
\end{enumerate} 
Evaluating Equation \eqref{Eq3.4} for $f_1', f_2'$ and $f_3'$ we get,
\begin{equation} \label{Eq3.5}
     (V^3-mI_E) v - (1+l) (\varphi^2)'(z)V^2 (v) + (l+m) \varphi'(z)V(v) =0,
\end{equation}
\begin{equation} \label{Eq3.6}
     (V^3-mI_E) zv - (1+l) (\varphi^2(z) (\varphi^2)'(z))V^2 (v) + (l+m) (\varphi(z)\varphi'(z)) V(v) = 0,
\end{equation}
and 
\begin{equation} \label{Eq3.7}
     (V^3-mI_E) z^2v - (1+l) ((\varphi^2(z))^2 (\varphi^2)'(z))V^2 (v) + (l+m) ((\varphi(z))^2\varphi'(z)) V(v) = 0, 
\end{equation}
respectively. 
Multiplying $\varphi(z)$ into Equation \eqref{Eq3.5} and subtracting from Equation \eqref{Eq3.6} we get
\begin{equation} \label{Eq3.8}
     (V^3-mI_E) (z-\varphi(z))v - (1+l) (\varphi^2)'(z)(\varphi^2(z)-\varphi(z))V^2 (v)=0.
\end{equation}
Again multiplying $(\varphi(z))^2$ into Equation \eqref{Eq3.5} and substracting from Equation \eqref{Eq3.7} we get,
\begin{equation} \label{Eq3.9}
     (V^3-mI_E) (z^2-(\varphi(z))^2)v - (1+l) (\varphi^2)'(z)((\varphi^2(z))^2-(\varphi(z))^2)V^2 (v)=0.
\end{equation}
Rewriting Equations \eqref{Eq3.8} and \eqref{Eq3.9} we get, 
\begin{equation*}
     (V^3-mI_E) (z-\varphi(z))v = (1+l) (\varphi^2)'(z)(\varphi^2(z)-\varphi(z))V^2 (v).
\end{equation*}
Again multiplying $(\varphi(z))^2$ into Equation \eqref{Eq3.5} and substracting from Equation \eqref{Eq3.7} we get,
\begin{equation*}
     (V^3-mI_E) (z-\varphi(z)) v = (1+l) (\varphi^2)'(z) \left(\frac{(\varphi^2(z))^2-(\varphi(z))^2}{ z+\varphi(z)} \right) V^2 (v).
\end{equation*}
Solving these two Equations, we get 
\begin{equation*}
    (1+l) (\varphi^2)'(z)[(\varphi^2(z)-\varphi(z))(\varphi^2(z)+\varphi(z))- (z+\varphi(z))(\varphi^2(z)-\varphi(z))]V^2 (v)=0.
\end{equation*} 
\begin{equation*}
    \implies (1+l) (\varphi^2)'(z)(\varphi^2(z)-\varphi(z))[\varphi^2(z)- z] V^2 (v)=0.
\end{equation*}
Now since $\varphi(z) \neq z, \varphi^2(z)\neq z, (\varphi^2)'(z) \neq 0$ because of the form of $\varphi$ and $V$ is an isometry, so $1+l=0$. 

Again using $1+l=0$ in Equations \eqref{Eq3.5} and \eqref{Eq3.6} we get, 
\begin{equation} \label{Eq3.10}
     (V^3-mI_E) v + (l+m) \varphi'(z)V(v) =0.
\end{equation}
\begin{equation} \label{Eq3.11}
     (V^3-mI_E) zv + (l+m) (\varphi(z)\varphi'(z)) V(v) = 0.
\end{equation}
Multiplying $z$ into Equation \eqref{Eq3.10} and substracting from Equation \eqref{Eq3.11} we get, 
\begin{equation}
 (l+m) \varphi'(z)(\varphi(z)-z) V(v) = 0.   
\end{equation}
Now since $\varphi(z) \neq z, \varphi'(z) \neq 0$ because of the form of $\varphi$ and $V$ is an isometry, so $l+m=0$.

This gives $l=-1, m=1$, so $\lambda_1, \lambda_2$ are the cube roots of unity. Moreover, $V^3=I_E$.
        
This completes the proof of the Theorem.
\end{proof}

Now we let $\mathcal{C}= \{ P_0, P_1, P_2\}$ be a family of $GTPs$ on $\mathcal{B}_{\ast}(\mathbb{D},E)$. The form of the surjective isometries on $\mathcal{B}_{\ast}(\mathbb{D}, E)$ given in \ref{Bstar-isom} implies that $P_i$ leaves invariant the subspace of all constant functions and also $\mathcal{B}_0(\mathbb{D}, E)$ for every $i=0,1,2$. Thus, the restriction of $P_i$ on $\mathcal{B}_0(\mathbb{D}, E)$ is a generalized tri-circular projection on $\mathcal{B}_0(\mathbb{D}, E)$. Applying Theorem \ref{main2} we get the form of $GTPs$ on $\mathcal{B}_{\ast}(\mathbb{D},E)$. Hence, we have the following Theorem: 

\begin{thm} \label{main4}
Let $\mathcal{C}= \{ P_0, P_1, P_2\}$ be a family of projections on
$\mathcal{B}_{\ast}(\mathbb{D}, E)$. Then $\mathcal{C}$ is a family of generalized tri-circular projections if and only if one of the following holds: 
\begin{enumerate}
\item $\varphi(z)=z$. In this case, $P_if(z) = P_i^{E}(f(z)), i=0,1,2$, and the collection $\mathcal{C}_E =\{ P_0^{E}, P_1^{E}, P_2^{E}\}$ is a family of $GTPs$ on $E$.
\item $\varphi(z)\neq z, \varphi^2(z)=z$. In this case, one of the members from $\mathcal{C}$ becomes a generalized bi-circular projection.
\item $\varphi(z) \neq z, \varphi^2(z)\neq z$ and $\varphi^3(z) = z$. Then $\lambda_1$, $\lambda_2$ are cube roots of unity and $U^3=I_E=V^3$.
\end{enumerate}
\end{thm} 

\subsection{Conflict of interest} There are no conflicts of interest in this paper.
%%%%%%%%%%%%%%%%%%%%%%%%%%%%%%%%%%%%%%%%%%%%%%%%%%
\bibliographystyle{MyOwn}
\bibliography{Citation}

\begin{thebibliography}{10}

\bibitem{abu2016generalized}
A.B.A. Baker.
\newblock Generalized 3-circular projections for unitary congruence invariant
  norms.
\newblock Banach J. Math. Anal., {\textbf{10}} (2016), no.~3, pages 451--465.

\bibitem{abubaker2016structures}
A.B.A. Baker and S.~Dutta.
\newblock Structures of generalized 3-circular projections for symmetric norms.
\newblock Proc. Indian Acad. Sci. Math. Sci., {\textbf{126}} (2016), pages
  241--252.

\bibitem{behrends2006m}
E.~Behrends.
\newblock {M}-structure and the {B}anach-{S}tone theorem.
\newblock Springer.

\bibitem{botelho2009generalized}
F.~Botelho and J.~Jamison.
\newblock Generalized bi-circular projections on spaces of analytic functions.
\newblock Acta Sci. Math. (Szeged), {\textbf{75}} (2009), no.~3, page 527.

\bibitem{botelho2014isometries}
F.~Botelho and J.~Jamison.
\newblock Isometries on the vector valued little {B}loch space.
\newblock IIllinois J. Math., {\textbf{58}} (2014), no.~3, pages 629--646.

\bibitem{cambern1990multipliers}
M.~Cambern and K.~Jarosz.
\newblock Multipliers and isometries in ${H}^\infty({E})$.
\newblock Bull. Lond. Math. Soc., {\textbf{22}} (1990), no.~5, pages 463--466.

\bibitem{vcuka2016generalized}
J.~{\v{C}}uka and D.~Ili{\v{s}}evi{\'c}.
\newblock Generalized tricircular projections on minimal norm ideals in
  ${B}({H})$.
\newblock J. Math. Anal. Appl., {\textbf{434}} (2016), no.~2, pages 1813--1825.

\bibitem{fovsner2007g}
M.~Fo{\v{s}}ner, D.~Ili{\v{s}}evi{\'c}, and C.K. Li.
\newblock G-invariant norms and bicircular projections.
\newblock Linear Algebra Appl., {\textbf{420}} (2007), no.~2-3, pages 596--608.

\bibitem{hosseini2022note}
M.~Hosseini.
\newblock A note on generalized tricircular projections on spaces of bounded
  variation functions.
\newblock Complex Anal. Oper. Theory, {\textbf{16}} (2022), no.~4, page~57.

\bibitem{hosseini2021projections}
M.~Hosseini and A.~Jim{\'e}nez-Vargas.
\newblock Projections in the convex hull of three isometries on absolutely
  continuous function spaces.
\newblock Bull. Malays. Math. Sci. Soc., {\textbf{44}} (2021), no.~5, pages
  3489--3509.

\bibitem{ilivsevic2022generalized}
D.~Ili{\v{s}}evi{\'c}, C.K. Li, and E.~Poon.
\newblock Generalized circular projections.
\newblock J. Math. Anal. Appl., {\textbf{515}} (2022), no.~1, page 126378.

\bibitem{lin1990isometries}
P.K. Lin.
\newblock The isometries of ${H}^\infty({E})$.
\newblock Pacific J. Math., {\textbf{143}} (1990), pages 69--77.

\bibitem{lin1991isometries}
P.K. Lin.
\newblock The isometries of ${H}^p({K})$.
\newblock J. Aust. Math. Soc., {\textbf{50}} (1991), no.~1, pages 23--33.

\bibitem{lin2008generalized}
P.K. Lin.
\newblock Generalized bi-circular projections.
\newblock J. Math. Anal. Appl., {\textbf{340}} (2008), no.~1, pages 1--4.

\bibitem{maurya2024automorphisms}
R.~Maurya, J.~Sarkar, and A.~Sensarma.
\newblock Automorphisms and generalized projections on spaces of analytic
  functions.
\newblock J. Math. Anal. Appl., {\textbf{530}} (2024), no.~2, page 127698.

\end{thebibliography}

\end{document}